\documentclass[11pt,a4paper,thmsb,ukenglish,fleqn]{article}
\usepackage{epsf,  amsmath, amssymb, graphicx}
\usepackage{amsthm}
\usepackage[sort]{natbib}
\usepackage{a4wide}
\usepackage{subfigure}
\usepackage{xcolor}
\usepackage{fancyhdr}
\newcommand{\mycomment}[1]{}

\newcommand{\R}{\ensuremath{\Bbb{R}}}
\newcommand{\N}{\ensuremath{\Bbb{N}}}

\newcommand{\Var}{\ensuremath{\text{Var}}}
 \newcommand{\Cor}{\ensuremath{\text{Cor}}}
 \newcommand{\Cov}{\ensuremath{\text{Cov}}}

\newcommand{\leb}{\ensuremath{Leb}}
\def\Levy{L\'{e}vy }
\def\levy{L\'{e}vy}

\newtheorem{theorem}{Theorem}

\newtheorem{condition}[theorem]{Condition}

\newtheorem{corollary}{Corollary}

\newtheorem{definition}[theorem]{Definition}
\newtheorem{example}{Example}

\newtheorem{proposition}{Proposition}
\newtheorem{remark}{Remark}

\allowdisplaybreaks
\begin{document}
\title{Stationary and multi-self-similar random fields with stochastic volatility}
\author{\textsc{Almut E.~D.~Veraart} \\
\textit{Department of Mathematics, Imperial College London}\\
\textit{ 180 Queen's Gate, 
 London, SW7 2AZ, 
UK}  \\
\texttt{a.veraart@imperial.ac.uk}
}
\date{Version: October 30, 2013}
\maketitle

\begin{abstract}
This paper introduces   stationary and multi-self-similar random fields which account for  stochastic volatility  and have type G marginal law.
The stationary random fields are constructed using volatility modulated mixed moving average fields and their probabilistic properties are discussed. Also,  two methods for parameterising the weight functions in the moving average representation are presented: One method is based on Fourier techniques and  aims at reproducing a given correlation structure, the other method is based on ideas from  stochastic partial differential equations.
 Moreover, using a generalised Lamperti transform we construct  volatility modulated multi-self-similar random fields which  have  type G distribution.
\end{abstract}

\maketitle{}
\section{Introduction}
Stationary infinitely divisible stochastic processes and random fields have been widely studied in the probability literature and have been found to constitute important building blocks for the stochastic modelling 
 of a wide range of  empirical phenomena. While influential theoretical work on infinitely divisible distributions and processes has to a great extent been established  in the  1970s and 1980s,  see e.g.~\cite{Sato1999} and \cite{SteutelvanHarn2004} for recent textbook treatments, 
the recent probability literature has taken up this topic again - not least due to research questions arising in the context of financial applications or in modelling of  turbulence in physics.
A recent review on this topic and related results can be found in 
\cite{BN2011}.

This paper focuses on 
(strictly) 
 stationary   random fields, which are parameterised as 
  so-called \emph{mixed moving average} (MMA) fields.
More precisely, consider a real-valued random field $X=(X(t))_{t\in \mathbb{R}^d}$ for $d\in \mathbb{N}$.
 An MMA field is given by 
\begin{align*}
X(t) &= \int_{\mathcal{X} \times\mathbb{R}^d}g(x,t-s)M(dx,ds), \quad t \in \mathbb{R}^d,
\end{align*}
where $\mathcal{X}$ is a subset of the Euclidean space $\mathbb{R}^k$ for $k\in \mathbb{N}$ and  $g: \mathcal{X} \times \mathbb{R}^d \to \mathbb{R}$ is a measurable deterministic function and 
$M$ is a \Levy basis, i.e.~an independently scattered, infinitely divisible  random measure. Random fields of such a type have for instance been studied in the context of  stable MMA fields
 by \cite{Surgailisetal1993}, also \cite{BNStelzer2010b} consider  \emph{supOU} processes, and  \cite{BN2011, BNLSV2012} study \emph{trawl} processes which fall into the class of MMA fields.

Motivated by the aforementioned  literature, but recognising the fact that such basic MMA models lack an important component which is relevant in many empirical studies, this paper concentrates on  the class of  mixed moving average fields which allow for \emph{stochastic volatility}. 
In particular, we will propose to replace the \Levy basis $M$ above by a volatility modulated Gaussian \Levy basis of the form 
\begin{align}\label{withSV}
M(dx, ds) = \sigma(s) W(dx, ds),
\end{align}
where $\sigma$ denotes a stochastic volatility field and $W$ a Gaussian \Levy basis.
Such random fields embed certain types of ambit fields, see  \cite{BNBV2012RecAdv}, which have recently been introduced in the literature.

While working with a Gaussian L\'{e}vy basis is very appealing from a mathematical point of view, many real world phenomena are not Gaussian and   we often need to account for distributions with   (semi)-heavy tails. A natural starting point for allowing for stochastic volatility and non-Gaussian distributions, is  to study distributions of \emph{type G}, see e.g.~\cite{Marcus1987,Rosinski1991}. A  distribution is of  type G if it is a variance mixture of a normal random variable with an independent  infinitely divisible mixing variable. 
This is a very wide class of distributions, which e.g.~includes the symmetric stable and the symmetric generalised hyperbolic distributions,   and is a natural starting point for studying non-Gaussian, volatility modulated  processes. 
While 
L\'{e}vy processes with type G distribution have been studied in detail in e.g.~\cite{Rosinski1991} and \cite{BNPerezAbreu1999}, this paper focuses on processes and random fields which are obtained through a variance mixture of a Gaussian L\'{e}vy basis as in \eqref{withSV} --  which will be made precise in the following section.

In particular,  this paper will introduce stationary volatility modulated MMA fields whose marginal distribution is of type G. 
We will study the probabilistic properties of such processes in detail.

While an MMA model appears to be rather general, in practical applications we often wish to parameterise the weight function $g$. Hence  we  introduce two methods for finding suitable parametric models for the  weight function $g$. First we start off  from the perspective of a given covariance function and  study which weight function can induce a given covariance function. Next, we will discuss how the weight function can be linked to a Green's function in certain types of stochastic partial differential equations, which could be taken as an alternative route for model building.

Stationary stochastic processes are of great importance in their own right, but in this paper we also use them as a tool for constructing \emph{(multi-) self-similar} random fields which exhibit stochastic volatility.  
Self-similar stochastic processes have been studied in great detail in the last five decades since the law of many empirical phenomena appears to be invariant under suitable temporal or spatial scaling. Relevant examples can for instance be found in climatology, hydrology,  turbulence,  network traffic, and  in economics.
Having in mind that stationary processes can be linked to self-similar processes via the Lamperti transform, see \cite{Lamperti1962}, we use the so-called \emph{generalised Lamperti transform}, see \cite{GPT2007},  to construct random fields which are \emph{multi-self-similar}, allow for stochastic volatility and whose distribution is also of type G. 
To the best of our knowledge, this is  first paper which studies 
 stochastic volatility modulation  and multi-self-similarity simultaneously.

The remaining part of the article is structured as follows. Section 
\ref{SectPre} reviews the relevant background material on \Levy bases.
The class of volatility modulated mixed moving average fields is defined in Section \ref{SectVMMMA} and its probabilistic properties are studied in detail. Next, Section \ref{SectCor} presents two methods for finding  relevant parametrisations of the weight function. Moreover,  we construct multi-self-similar random fields with stochastic volatility in Section \ref{SectMSS} and, finally, Section \ref{SectCon} concludes. 

 \section{Preliminaries}\label{SectPre}
 In the following we briefly review basic definitions and well-known facts on  \Levy bases, cf.~\cite{RajRos89}, \cite{P} and \cite{BNBV2012RecAdv} for details.

 Let $(\Omega, \mathcal{F}, P)$ denote  a  probability space and 
$(S,\mathcal{S}, \leb)$  a Lebesgue-Borel space; here  
 $S$ denotes a 
 Borel set in 
$\mathbb{R}^m$ for a $m \in \mathbb{N}$; a typical choice would be  $S= \mathbb{R}^m$. Also, we denote by   $\mathcal{S}=\mathcal{B}(S)$  the Borel $\sigma$-algebra on $S$ and by $\leb$  the Lebesgue measure.
Moreover, we define the $\delta$-ring
$$\mathcal{B}_b(S)= \{A \in \mathcal{S}: \leb(A) < \infty \},$$
which is the subset of $\mathcal{S}$ that contains sets which have bounded Lebesgue measure. 

Recall that a \Levy basis is defined as an independently scattered random measure $L = \{L(A): A \in \mathcal{B}_b(S)\}$ on $\mathcal{B}_b(S)$, such that for every  $A \in \mathcal{B}_b(S)$, $L(A)$ is infinitely divisible (ID) with characteristic function
\begin{multline}\label{LevKhi}
\mathbb{E}(\exp(i \theta L(A))
\\ = \exp\left(i \theta a^*(A) - \frac{1}{2}\theta^2 b^*(A) + \int_{\mathbb{R}}\left(e^{i\theta x}-1-i\theta  x \mathbb{I}_{[-1,1]}(x)\right) n(dx, A)\right), 
\end{multline}
for $\theta \in \R$. Here 
$a^*$ is a signed measure on $\mathcal{B}_b(S)$, $b^*$ is a measure on $\mathcal{B}_b(S)$, and 
$n(\cdot,\cdot)$ is the generalised L\'{e}vy measure, meaning that $n(dx, A)$ and a measure on $\mathcal{B}_b(S)$ for fixed $dx$ and  a L\'{e}vy measure on $\mathbb{R}$ for fixed $A \in \mathcal{B}_b(S)$.
Define the measure $c$ by
\begin{align}\label{ControlM}
c(A) = |a^*|(A) + b^*(A) + \int_{\R}\min(1,x^2)n(dx,A),\qquad A \in \mathcal{B}_b(S).
\end{align}
Following \citet[Proposition 2.1 (c), Definition 2.2]{RajRos89}, we define the
 \emph{control measure} as   the extension of the  measure 
$c$ to 
  a $\sigma$-finite measure on 
$(S, \mathcal{S})$, which we will also denote by $c$. It is often useful to employ an infinitesimal notation, as given below. 
We define  
  the Radon-Nikodym derivatives of the three components of $c$ by
 \begin{align}\label{CQ}
 a(z) = \frac{da^*}{dc}(z), &&
 b(z) = \frac{db^*}{dc}(z), &&
 \nu(dx, z) = \frac{n(dx,\cdot)}{dc}(z),
 \end{align}
 where we will assume w.l.o.g.~that $\nu(dx, {\bf z})$ is a \Levy measure for each fixed ${\bf z}$.
 We call  $(a, b, \nu(dx, \cdot), c)=(a(z), b(z), \nu(dx, z), c(dz))_{z\in S}$ the  \emph{characteristic quadruplet} (CQ) associated with the \Levy basis $L$. 
Typically we work with  \emph{dispersive} L\'{e}vy bases, which satisfy   $c(\{z\})=0$ for all $z\in S$.

Note also that we have
\begin{align}\begin{split}\label{LevKhii}
&\mathbb{E}(\exp(i \theta L(dz))
\\
& = \exp\left( \left(i \theta a(z) - \frac{1}{2}\theta^2 b(z) + \int_{\mathbb{R}}\left(e^{i\theta x}-1-i\theta x \mathbb{I}_{[-1,1]}(x)\right) \nu(dx,z)\right)c(dz) \right)\\
&= \exp\left( K(\theta, L'(z)) c(dz)\right), \qquad \theta \in \R,
\end{split}\end{align}
where we call $L'(z)$  the \emph{Levy seed} of $L$ at $z$, which is defined as the infinitely divisible random variable having \levy-Khintchine representation 
\begin{align}\label{CharL'}
 \mathbb{E}(i \theta L'(z )) &= \exp(K(\theta, L'(z))),\\
K(\theta, L'(z)) & =  i \theta a(z) - \frac{1}{2}\theta^2 b(z) + \int_{\mathbb{R}}\left(e^{i\theta x}-1-i\theta x \mathbb{I}_{[-1,1]}(x)\right) \nu(dx, z).
\end{align}
Finally, recall that 
if $\nu(dr, z)$ does not depend on $z$, we call  $L$ \emph{factorisable}. 
If, in addition, 
 $c$ is also proportional to the Lebesgue measure and $a(z)$ and $b(z)$ do not depend on $z$, then $L$ is called \emph{homogeneous}. 

 In the following, we will define integrals with respect to L\'{e}vy bases, where we use the integration concept developed in  \cite{RajRos89} when we are dealing with deterministic integrands. In particular, let $f: (S, \mathcal{S}) \to (\mathbb{R}, \mathcal{B}(\mathbb{R}))$ denote a measurable function. According to \citet[Theorem 2.7]{RajRos89}, $f$ is integrable with respect to $L$ if and only if
 \begin{align}\begin{split}
\label{GenIntCond}
 \int_{S} \left| f(s) a(s) + \int_{-\infty}^{\infty} \left(\mathbb{I}_{[-1,1]}(w f(s))- f(s)\mathbb{I}_{[-1,1]}(w) \right)\nu(dw,s)\right| c(ds) < \infty,\\
 \int_S|f(s)|^2b(s) c(ds) < \infty,
 \\
 \int_{S}  \int_0^{\infty}\min(1, |f(s)w|^2)\nu(dw,s) c(ds) < \infty.
  \end{split}
\end{align}
\section{Volatility modulated mixed moving average fields}\label{SectVMMMA}

Let $S=\mathcal{X} \times \mathbb{R}^d$ for $d\in \mathbb{N}$ and let $W$ denote a standard  Gaussian independently scattered random measure with characteristic quadruplet $(0,1, 0, c)$.
In particular, the characteristic function of $W$ is given by
$\mathbb{E}(\exp(i \theta W(A))) = \exp\left(-\frac{1}{2}\theta^2c(A)\right)$, $\forall A\in \mathcal{B}_b(S)$.
In the following, we will assume that $c(dz) = p(dx)ds$ for a probability measure $p$.

A \emph{volatility modulated mixed moving average} (VMMMA) field is defined by 
\begin{align}\label{VMMMA}
X(t) &= \int_{\mathcal{X} \times\mathbb{R}^d}g(x,t-s)\sigma(s)W(dx,ds), \quad t \in \mathbb{R}^d,
\end{align}
where  $g: \mathcal{X} \times \mathbb{R}^d \to \mathbb{R}$ is a measurable deterministic function and 
 $\sigma: \mathbb{R}^d \times \Omega \to \mathbb{R}$ is a strictly stationary random field independent of $W$. 
Note that throughout the paper, we write $t=(t_1, \dots, t_d)^{\top}$, where we typically 
  interpret the parameter $t_1$ as the \emph{time} parameter, and the remaining parameters $(t_2, \dots, t_d)$ as \emph{space} parameters.
  
  Since the integrand in \eqref{VMMMA} is stochastics, we cannot work with the \cite{RajRos89} integration concept to define the integral in  \eqref{VMMMA}. Instead we use the theory of \cite{W}, who developed and integration theory with respect to orthogonal martingale measures. A detailed description and review of this integration theory can be found in \cite{BNBV2011, BNBV2012RecAdv}.

To this end, we define the filtration
    as follows. First of all, we separate out the time parameter in the \Levy basis, which we denote by $t_1$. Then, we define 
$W_{t_1}(A)= W([0,t_1]\times A)$ for a measurable set $A \subset S:=\mathcal{X}\times \R^{d-1}$. Let  
\begin{align}\label{filt}
\mathcal{F}_{t_1} = \cap_{n=1}^{\infty} \mathcal{F}_{t_1+1/n}^0,&& \text{where}\qquad  \mathcal{F}_{t_1}^0 =\sigma \{W_s(A): A \in \mathcal{B}_b(S), 0 < s \leq t_1 \}\vee \mathcal{N},
\end{align}
and 
where $\mathcal{N}$ denotes the $P$-null sets of $\mathcal{F}$. 
We see  that  $\mathcal{F}_{t_1}$ is right-continuous by construction. 
  Following \cite{W}, one can then define a stochastic integral with respect to the Gaussian \Levy basis, where we require that the integrand is square-integrable and predictable in the time-component. More precisely, we need the following condition.

\begin{condition}\label{I}
Suppose that $\sigma=(\sigma(s))_{s\in \R^d}$ is predictable in the first component $s_1$ and that
\begin{align}\label{I2Cond}
\int_{\mathcal{X}\times \mathbb{R}^d}  g^2(x,t-s)\mathbb{E}\left(\sigma^2(s)\right)p(dx)ds < \infty.
\end{align}
\end{condition}

\begin{proposition}[Existence]
The random field $X$ is well-defined provided equation  \eqref{I2Cond} holds. Also, $X$ is 
 strictly stationary, and, since it is  square integrable, $X$ is also second-order stationary.
\end{proposition}
\begin{proof}
This is an immediate consequence of \cite{W}, see also \cite{BNBV2012RecAdv} for details.
\end{proof}

\subsection{Examples}
Note that if the space $\mathcal{X}$ consists of only one point, or if  $p$ has only one atom, then we obtain the (volatility modulated) moving average field, cf.~\cite{Surgailisetal1993}, \citet[p.~591-592]{STaqqu1994},
\begin{align}\label{MA}
X(t) &= \int_{\mathbb{R}^d}g(t-s)\sigma(s)W(ds), \quad t \in \mathbb{R}^d.
\end{align}
When we study the \emph{mixed} case,  we think of 
$x$  as a parameter (vector) of the weight function $g$, which  can  be randomised through the L\'{e}vy basis.
I.e.~in terms of choices of the weight function, we could essentially choose the same ones as in the classical moving average case and then randomise some or all of the parameters in the weight function. 

\begin{example} 
As a first example, we consider supOU processes with stochastic volatility. They are obtained by choosing $d=1$, $\mathcal{X}=\mathbb{R}$ and  $g(x, t-s)=\exp(-x (t-s))\mathbb{I}_{[0,\infty)}(t-s)$. Clearly, an Ornstein-Uhlenbeck (OU) process is a special case of a supOU process. When we randomise the mean-reversion parameter in the OU process, here denoted by $x$, we can obtain long memory processes, cf.~\cite{BN2000}.
\end{example}

\begin{example}
When $g(x,t-s):=\mathbb{I}_{A(0)}(s-t)f(x,t-s)$ for a measurable set $A(0)\subset (-\infty,0] \times \mathbb{R}^{d-1}$ and $A(t) = A(0)+t$, we obtain a stationary mixed ambit field of the form
\begin{align*} 
 X(t) &= \int_{\mathcal{X} \times A(t)}f(x,t-s)\sigma(s)W(dx,ds).
\end{align*}
Ambit fields have been introduced to model tempo-spatial phenomena such as turbulence, cell growth and financial futures, see \cite{BNBV2012RecAdv} for a recent survey, and they constitute an analytically tractable alternative to modelling by stochastic (partial) differential equations.   
Note that in  applications one often needs an additional drift term which we ignore in this paper to simplify the exposition.
\end{example}

As soon as we remove the stochastic volatility component, we  are back to classes of stochastic processes which have been studied in the literature. Related random fields and stochastic processes in the absence of stochastic volatility include a L\'{e}vy-driven mixed moving average field of the form
\begin{align*}
X(t) &= \int_{\mathcal{X}\times \mathbb{R}^d} g(x, t-s)L(dx, ds),
\end{align*}
where $L$ is a L\'{e}vy basis. Clearly, $X$ is a mixed moving average field as defined (in the context of symmetric $\alpha$ stable random measures) by \cite{Surgailisetal1993}, see also \cite{Fasen2005, MoserStelzer2013} for the case when $t \in \mathbb{R}$. 
Also, let  $A(0)\subset (-\infty,0] \times \mathbb{R}^{d-1}$ denote  a measurable set (as before). Then  the stochastic process $X=(X(t))_{t\in\mathbb{R}}$ defined by 
\begin{align*}
X(t) &=
\int_{\mathcal{X}\times \mathbb{R}^d} \mathbb{I}_{A(0)}(x,s-t)L(dx,ds),
\end{align*}
is a so-called \emph{trawl process}, cf.~\cite{BNLSV2012,BNBV2012RecAdv}.

In Section \ref{ginplane}  we will discuss some possible choices of the weight function $g$, which go beyond the  choices of  OU, supOU or trawl   weight functions.

\subsection{Probabilistic properties of VMMMA fields}\label{SectProbVMMMA}
Let $\mathcal{F}^{\sigma}$ denote the $\sigma$-algebra generated by the random field $\sigma$. 
Then, since we assume that  $\sigma$ is independent of $W$, we have
\begin{align*}
X(t)|\mathcal{F}^{\sigma} \sim N\left(0, V(t)\right), 
\end{align*}
where 
\begin{align}\label{V}
 V(t)=\int_{\mathcal{X} \times \mathbb{R}^d}g^2(x,t-s)\sigma^2(s)p(dx)ds.
\end{align}
We observe that $V=(V(t))_{t \in \R^d}$ has the nature of an integrated weighted volatility field. 
Recall that in the financial econometrics and mathematical finance literature, integrated volatility  
of the form
\begin{align*}
\int_0^t\sigma(s) ds, \quad t\geq 0,
\end{align*}
is a key quantity of interest since it constitutes accumulated stochastic variance over a time period $[0, t]$, see e.g.~\cite{BNS}.
In the context of the random fields we study here, where $s\in \R^d$ is multivariate, we obtain a related quantity, where we do not just accumulate over the time parameter, but also over a (possibly multi-dimensional) space parameter. Since we are not necessarily integrating over a bounded interval (unless the weight function $g$ contains a suitable indicator function), the weight function $g$ plays the role of down-weighting the volatility for points $s$ far away in time and/or space.

Note that  the conditional characteristic function of $X$ can be expressed in terms of the accumulated volatility $V$ and  is given by
\begin{align}\label{CondChar}
\mathbb{E}(\exp(i \theta X(t))|\mathcal{F}^{\sigma})
=\exp\left(-\frac{1}{2}\theta^2V(t)\right).
\end{align}

\subsubsection{Properties of  the stochastic variance}
The volatility field $\sigma$ can be chosen in many different ways. Here 
we are interested in specifications which lead to an \emph{infinitely divisible} variance process and hence choose the following parametrisation.
\begin{condition}\label{V1}
 Suppose the stochastic volatility field is given by a mixed moving average of the form
\begin{align*}
\sigma^2(s) = \int_{\mathcal{X}^{\sigma} \times\mathbb{R}^d}h(y,s-u)L^{\sigma}(dy,du),
\end{align*}
where $L^{\sigma}$ is a L\'{e}vy basis with a L\'{e}vy seed given by a square-integrable subordinator with CQ $(a^{\sigma},0, \nu^{\sigma}, p^{\sigma}\otimes \leb)$,
where we assume that  $\delta =a^{\sigma}-\int_{|w|\leq 1}w\nu^{\sigma}(dw)= 0$ and 
 where $p^{\sigma}$ is a probability measure and $h$ is a positive weight function satisfying the integrability conditions given in \cite{RajRos89}, i.e.~
 \begin{multline}
\label{IntCond}
 \int_{\mathcal{X^{\sigma}}\times \mathbb{R}^d} \left| h(y,z) a^{\sigma} + \int_0^{\infty} \left(\mathbb{I}_{[-1,1]}(w h(y,z))- h(y,z)\mathbb{I}_{[-1,1]}(w) \right)\nu^{\sigma}(dw)\right| p^{\sigma}(dy)dz < \infty,\\
 \int_{\mathcal{X}^{\sigma}\times \mathbb{R}^d}  \int_0^{\infty}\min(1, |h(y,z)w|^2)\nu^{\sigma}(dw) p^{\sigma}(dy)dz < \infty.
\end{multline}
\end{condition}

Note that the moment generating function of $\sigma^2$ (provided it exists) is given by
\begin{align}\nonumber
 \mathbb{E}(\exp(\theta \sigma^2(s)))&= \exp(K_{L^{\sigma}}(\theta)), 
 \,\,\,\,
 \text{where } \\
 \,  
K_{L^{\sigma}}(\theta)
&=\int_{\mathcal{X}^{\sigma}\times \mathbb{R}^d}  \int_0^{\infty}\left(e^{\theta h(x,z) w} -1 \right)\nu^{\sigma}(dw)p^{\sigma}(dx)dz, \label{K}
\end{align}
denotes the kumulant generating function of the L\'{e}vy seed associated with $L^{\sigma}$.

In the following, we will derive a representation result for the variance term which appears in the variance mixture with the Gaussian L\'{e}vy basis.

\begin{condition}\label{V2}
Assume that Condition \ref{V1} holds. Further, assume that for all $y \in \mathcal{X}, u \in \mathbb{R}^d$ we have
\begin{align*}
k(y,t-u):= \int_{\mathcal{X}\times \mathbb{R}^d}g^2(x,t-s)h(y,s-u)p(dx)ds < \infty,
\end{align*}
where $k(y,z)= (\widetilde g* h(y, \cdot))(z)$ is the convolution of $\widetilde g$ and $h$ with 
 $\widetilde g(z):=\int_{\mathcal{X}}g^2(x,z)p(dx)$. Also, we assume that $k$ satisfies the integrability conditions \eqref{IntCond} when we replace $h$ by $k$. 
\end{condition}
\begin{proposition}
Under Condition \ref{V2}, the stochastic variance field $V$ can be represented as
\begin{align*}
V(t) &= \int_{\mathcal{X}^{\sigma} \times \mathbb{R}^d}k(y,t-u) L^{\sigma}(dy,du).
\end{align*}
Also, 
$V$ is stationary and its marginal distribution is infinitely divisible.
Moreover, the Laplace transform is given by 
\begin{align}\label{LaplaceV}
\mathbb{E}(\exp(-\theta V(t)))=\exp(\Lambda_V(\theta)), \quad \text{where } 
\Lambda_V(\theta)&= 
\int_0^{\infty}\left(e^{-\theta x} -1\right)U(dx),  \quad \theta > 0,
\end{align}
where $U(dx)=\int_{\mathbb{R}^d} \int_{\mathcal{X}^{\sigma} }\nu^{\sigma}(dx (k(y,z))^{-1}) p^{\sigma}(dy)dz$ is a L\'{e}vy measure on $[0, \infty)$.
\end{proposition}
\begin{proof}
We apply the stochastic Fubini theorem for L\'{e}vy bases, cf.~\citet[Theorem 3.1]{BNBasse2009} and obtain  
\begin{align*}
V(t) &=  \int_{\mathcal{X}^{\sigma}  \times \mathbb{R}^d}\left(  \int_{\mathcal{X}\times\mathbb{R}^d} g^2(x,t-s)h(y,s-u)p(dx)ds\right) L^{\sigma}(dy,du)\\
&= \int_{\mathcal{X}^{\sigma}  \times \mathbb{R}^d}k(y,t-u) L^{\sigma}(dy,du).
\end{align*}
Moreover, we have
\begin{align*}
\mathbb{E}(\exp(i\theta V(t)))
&= 
\exp\left(\int_{\mathbb{R}^d} \int_{\mathcal{X}^{\sigma}  }\int_0^{\infty}\left(e^{i\theta k(y,t-u) w} -1\right)\nu^{\sigma}(dw) p^{\sigma}(dy)du\right)\\
&= 
\exp\left(\int_{\mathbb{R}^d} \int_{\mathcal{X}^{\sigma}  }\int_0^{\infty}\left(e^{i\theta k(y,z) w} -1\right)\nu^{\sigma}(dw) p^{\sigma}(dy)dz\right)\\
&= 
\exp\left(\int_0^{\infty}\left(e^{i\theta x} -1\right)U(dx)\right), 
\end{align*}
where we define $U(dx)=\int_{\mathbb{R}^d} \int_{\mathcal{X}^{\sigma} }\nu^{\sigma}(dx (k(y,z))^{-1}) p^{\sigma}(dy)dz$. It is an easy exercise to show that $U$ is indeed a \Levy measure on $[0,\infty)$. The result for the Laplace transform follows.
The proof of the stationarity is  a straightforward computation and hence omitted.
\end{proof}

\subsubsection{Marginal distribution}
\begin{proposition}
Under Condition \ref{V2}, $X$ given by \eqref{VMMMA} is a stationary random field whose marginal distribution is  infinitely divisible and  belongs to the class of type G distributions. 
Moreover, 
\begin{align}\begin{split}\label{CharFct}
\mathbb{E}(\exp(i\theta X(t)))&= \exp(-\Psi(\theta^2/2)), \; \text{ where }\\
\Psi(\zeta)&=-\int_0^{\infty}\left(e^{-\zeta x} -1\right)U(dx)=-\Lambda_V(\zeta), \end{split}
\end{align}
where the function $\Psi$ satisfies $\Psi(0)=0$ and has a completely monotone derivative on $(0,\infty)$
\end{proposition}

\begin{proof}
Since $X$ is a variance mixture with an infinitely divisible process, $X$ has itself ID marginal distribution, cf.~\citet[p.~29]{Rosinski1991}. 
The fact that its marginal distribution is of type G follows directly from the definition of type G distribution, cf.~\cite{Marcus1987}. From \citet[Proposition 3]{Rosinski1991} we deduce the corresponding representation result for the characteristic function. More precisely, from equations  \eqref{CondChar} and \eqref{LaplaceV}  we immediately get that 
\begin{align*}
\mathbb{E}(\exp(i\theta X(t)))
&= 
\exp\left(\int_0^{\infty}\left(e^{-\frac{1}{2}\theta^2 x} -1\right)U(dx)\right)=\exp\left( \Lambda_V\left(\frac{\theta^2}{2}\right)\right). 
\end{align*}

So, we define
$\Psi(\zeta):=-\int_0^{\infty}\left(e^{-\zeta x} -1\right)U(dx)$.
Clearly, $\Psi(0)=0$ and the derivative is given by
$\Psi'(\zeta)=\int_0^{\infty}e^{-\zeta x} x U(dx)$.
Note that, by definition, the function $k$ is non-negative, so we can conclude that $\Psi'$ is indeed completely monotone  on $(0,\infty)$, since it possesses derivatives of all orders and $ (-1)^n\Psi^{(n+1)}\geq 0$, for $n \in \mathbb{N}$ and  $z > 0$,  cf.~\citet[p.~415]{Feller1966}.
\end{proof}
\subsubsection{Cumulants and correlation structure}
From the joint cumulant function, we can  derive the cumulants and correlation structure of a volatility modulated mixed moving average field.

Since $X(t)| \mathcal{F}^{\sigma} \sim N(0, V(t))$, we immediately get the following results for the conditional cumulants of $X(t)$:
\begin{align*}
\kappa_1^{\sigma}=\mathbb{E}(X(t)|\mathcal{F}^{\sigma}) = 0,
\quad
\kappa_2^{\sigma}=\Var(X(t)|\mathcal{F}^{\sigma}) &= V(t),
\quad 
\kappa_i^{\sigma}= 0, \; i \geq 3.
\end{align*}
Unconditionally, we get for the first two cumulants that 
\begin{align*}
\mathbb{E}(X(t)) = 0,
&&
\Var(X(t)) &= \mathbb{E}(V(t)).
\end{align*}
For $t, t^* \in \mathbb{R}^d$, the covariance structure is given - conditionally -- by
\begin{align*}
\Cov(X(t), X(t^*)|\mathcal{F}^{\sigma})=\mathbb{E}(X(t)X(t^*)|\mathcal{F}^{\sigma}) = \int_{\mathcal{X}\times \mathbb{R}^d}g(x,t-s)g(x, t^*-s)\sigma^2(s) p(dx) ds,
\end{align*}
and - unconditionally -- by
\begin{align*}
\Cov(X(t), X(t^*))=\mathbb{E}(X(t)X(t^*)) = \int_{\mathcal{X}\times \mathbb{R}^d}g(x,t-s)g(x, t^*-s)\mathbb{E}(\sigma^2(s)) p(dx) ds,
\end{align*}
which in the case of a second-order stationary volatility field simplifies to 
\begin{align*}
R_X(h) &:=  \Cov(X(h), X(0))= \mathbb{E}(\sigma^2(0)) \int_{\mathcal{X}\times \mathbb{R}^d}g(x,h+s)g(x, s) p(dx) ds,\\
\rho_X(h) &:=  \Cor(X(h), X(0))=  \frac{\int_{\mathcal{X}\times \mathbb{R}^d}g(x,h+s)g(x, s) p(dx) ds}{\int_{\mathcal{X}\times \mathbb{R}^d}g^2(x,s)p(dx) ds}.
\end{align*}
Interestingly, that means that the stochastic volatility component has no impact on the correlation structure. This changes, however, as soon as  higher order correlations are considered. E.g.~we have
\begin{align*}
\Cov(X^2(t), X^2(t^*)|\mathcal{F}^{\sigma})
=2\left( \int_{\mathbb{R}^d}\int_{\mathcal{X}} g(x,t-s)g(x,t^*-s)\sigma^2(s) p(dx) ds\right)^2,
\end{align*}
and unconditionally we have
\begin{multline}\label{CovX2}
\Cov(X^2(t), X^2(t^*))=
2\mathbb{E}\left( \int_{\mathbb{R}^d}\int_{\mathcal{X}} g(x,t-s)g(x,t^*-s)\sigma^2(s) p(dx) ds\right)^2\\ +
\Cov(V(t),V(t^*)).
\end{multline}
The above results are interesting since they suggest that, in practical applications, estimation of such models could follow a multi-step estimation procedure, where one uses the variogram or covariance function, see \cite{Cressie1993} to identify $g$, and then one uses a second order variogram or covariance function to identify $h$ and finally one can estimate the remaining parameters coming from the L\'{e}vy basis $L^{\sigma}$ using a method of moments or a (quasi-) likelihood approach.

\subsection{Finite dimensional distributions}
Next we study  the finite dimensional distributions of the random field $(X_t)_{t\in \mathbb{R}^d}$.
\begin{proposition}\label{FinDist}
Let $n\in \mathbb{N}$ and  $\theta_1, \dots, \theta_n \in \mathbb{R}$.
The conditional finite dimensional distributions of $(X_t)_{t\in \mathbb{R}^d}$ given $\mathcal{F}^{\sigma}$ are given by
\begin{align*}
\mathbb{E}\left.\left(\exp\left(i \sum_{j=1}^n \theta_j X(t_j) \right)\right| \mathcal{F}^{\sigma} \right)
=\exp\left(-\frac{1}{2}\int_{\mathcal{X}\times \mathbb{R}^d}\left(\sum_{j=1}^n\theta_j g(x,t_j-s)\right)^2 \sigma^2(s)p(dx) ds\right).
\end{align*}
The  finite dimensional distributions are given by
\begin{align*}
\mathbb{E}\left(\exp\left(i \sum_{j=1}^n \theta_j X(t_j) \right) \right)
=\mathbb{E}\left(\exp\left(-\frac{1}{2}\int_{\mathcal{X}\times \mathbb{R}^d}\left(\sum_{j=1}^n\theta_j g(x,t_j-s)\right)^2 \sigma^2(s)  p(dx) ds\right) \right).
\end{align*}
 Under Condition \ref{V2}, the  finite dimensional distributions are given by
\begin{multline*}
\mathbb{E}\left(\exp\left(i \sum_{j=1}^n \theta_j X(t_j) \right) \right)
\\= \exp\left(\int_{\mathcal{X} \times\mathbb{R}^d}K_{L^{\sigma}}\left(-\frac{1}{2}k\left(y,\theta, (t_k-t_j)_{j,k=1,\dots,n} ,w\right)\right)p^{\sigma}(dy) dw\right),
\end{multline*}
where $K_{L^{\sigma}}$ denotes the kumulant generating function of the L\'{e}vy seed associated with $L^{\sigma}$ defined in \eqref{K} and where $\theta=(\theta_1, \dots, \theta_d)^{\top}$ and 
\begin{align*}
k(y,\theta, (t_k-t_j)_{j,k=1,\dots,n} ,w)=\sum_{j,k=1}^n\theta_j \theta_k\int_{\mathcal{X}\times \mathbb{R}^d}g(x,v)g(x,t_k-t_j+v) h(y,-w-v)p(dx) dv.
\end{align*}
\end{proposition}
\begin{proof}
The first two results are  a direct consequence of  \citet[Proposition 3.4.2]{STaqqu1994}. The third result follows from an application of the stochastic Fubini theorem, more precisely note that 
\begin{align*}
&\mathbb{E}\left(\exp\left(i \sum_{j=1}^n \theta_j X(t_j) \right) \right)
=\mathbb{E}\left(\exp\left(-\frac{1}{2}\int_{\mathcal{X}\times \mathbb{R}^d}\left(\sum_{j=1}^n\theta_j g(x,t_j-s)\right)^2 \sigma^2(s) p(dx) ds\right) \right)\\
&= \mathbb{E}\left(\exp\left(-\frac{1}{2}\int_{\mathcal{X}\times \mathbb{R}^d}\left(\sum_{j=1}^n\theta_j g(x,t_j-s)\right)^2 \int_{\mathcal{X}^{\sigma} \times\mathbb{R}^d}h(y,s-u)L^{\sigma}(dy,du) p(dx) ds\right) \right)\\
&= \mathbb{E}\left(\exp\left(\int_{\mathcal{X}^{\sigma} \times\mathbb{R}^d}\left\{\int_{\mathcal{X}\times \mathbb{R}^d}-\frac{1}{2}\left(\sum_{j=1}^n\theta_j g(x,t_j-s)\right)^2 h(y,s-u)p(dx) ds \right\} L^{\sigma}(dy,du)\right) \right),
\end{align*}
where a change of variable argument leads to 
\begin{align*}
&\int_{\mathcal{X}^{\sigma} \times\mathbb{R}^d}\left(\int_{\mathcal{X}\times \mathbb{R}^d}\left(\sum_{j=1}^n\theta_j g(x,t_j-s)\right)^2 h(y,s-u)p(dx) ds  \right)L^{\sigma}(dy,du)\\
&=
\int_{\mathcal{X}^{\sigma} \times\mathbb{R}^d}\left(\int_{\mathcal{X}\times \mathbb{R}^d}\sum_{j,k=1}^n\theta_j \theta_kg(x,t_j-s)g(x,t_k-s) h(y,s-u)p(dx) ds \right )L^{\sigma}(dy,du)\\
&=
\int_{\mathcal{X}^{\sigma} \times\mathbb{R}^d}\left(\int_{\mathcal{X}\times \mathbb{R}^d}\sum_{j,k=1}^n\theta_j \theta_kg(x,v)g(x,t_k-t_j+v) h(y,t_j-v-u)p(dx) dv \right )L^{\sigma}(dy,du)\\
&=
\int_{\mathcal{X}^{\sigma} \times\mathbb{R}^d}k(y,\theta, (t_k-t_j)_{j,k=1,\dots,n} ,w)L^{\sigma}(dy,dw),
\end{align*}
where $\theta=(\theta_1, \dots, \theta_d)^{\top}$ and 
\begin{align*}
k(y,\theta, (t_k-t_j)_{j,k=1,\dots,n} ,w)=\sum_{j,k=1}^n\theta_j \theta_k\int_{\mathcal{X}\times \mathbb{R}^d}g(x,v)g(x,t_k-t_j+v) h(y,-w-v)p(dx) dv.
\end{align*}
Hence, we have
\begin{align*}
&\mathbb{E}\left(\exp\left(i \sum_{j=1}^n \theta_j X(t_j) \right) \right)\\
&= \mathbb{E}\left(\exp\left(-\frac{1}{2}\int_{\mathcal{X}^{\sigma} \times\mathbb{R}^d}k(y,\theta, (t_k-t_j)_{j,k=1,\dots,n} ,w)  L^{\sigma}(dy,du)\right) \right)\\
&=\exp\left(\int_{\mathcal{X} \times\mathbb{R}^d}K_{L^{\sigma}}\left(-\frac{1}{2}k(y,\theta, (t_k-t_j)_{j,k=1,\dots,n} ,w)\right)p^{\sigma}(dy) dw\right),
\end{align*}
where $K_{L^{\sigma}}$ denotes the kumulant function defined in \eqref{K}.
\end{proof}

\begin{remark}
It is an immediate consequence of Proposition \ref{FinDist} that $X$ is stationary if and only if $(\sigma(t))_{t\in \mathbb{R}^d}$ is stationary.
\end{remark}


\section{Parameterising  the weight function of a  VM(M)MA field}
 \label{SectCor}
 So far, we have only required that the weight function $g$ satisfies a square  integrability condition, but have not commented much on particular functional forms - apart from few examples given in Section \ref{SectVMMMA}. In this section, we will now discuss how parametric classes of weight functions can be derived which are relevant for applications.
 
 We distinguish two approaches: First, in the context of (tempo-) spatial models, many stochastic models focus on modelling the covariance function directly, see \cite{Cressie1993, CressieWikle2011} for textbook treatments. Motivated by this branch of the literature, we will show how a weight function can be constructed which reproduces a given correlation function.
Second, we will discuss that the weight function can be related to the Green's function in certain stochastic partial differential equations and we will study some concrete examples in the context of a VMMMA field on the plane. 
\subsection{Starting from  the covariance function}\label{SectRelMMAMA}
We start off by investigating the 
 relationship between the weight function $g$ and the covariance function $R_X$ via $L^2$-Fourier transforms. In particular, we have the following result.
\begin{proposition}\label{CorRep}
The autocorrelation  function of the VMMA field defined in  \eqref{VMMMA} satisfies 
\begin{align*}
\rho_X(h) &\propto  \frac{1}{(2\pi)^{d/2}}
\int_{\mathbb{R}^d}e^{ih^{\top}u}\gamma(u) du,
\end{align*}
for $\gamma \in L^1(\leb)$ with 
$\gamma(u):=\int_{\mathcal{X}}|\widehat g(x,u)|^2p(dx)$ and, for $u \in \mathbb{R}$,  $\widehat g(x, u)$ denotes the $L^2$-Fourier-transform of $g(x, \cdot)$. Also, $\gamma$ is proportional to the corresponding spectral density of $X$.
\end{proposition}
\begin{proof}
To simplify the exposition, we will in in the following assume that $\mathbb{E}(\sigma^2(0))=1$. 
Then 
\begin{align*}
R_X(h) &=   \int_{\mathcal{X}\times \mathbb{R}^d}g(x,h+s)g(x, s) p(dx) ds = \int_{\mathcal{X}}\left(\int_{\mathbb{R}^d}g(x,h+s)g(x, s)ds\right) p(dx).
\end{align*}
Recall that for $x\in \mathcal{X}$, we have that $g(x, \cdot) \in L^2(\leb)$. Now, for $u \in \mathbb{R}$ let $\widehat g(x, u)$ denote the Fourier-transform of $g(x, \cdot)$ in $L^2$.  Then we have
\begin{align*}
\int_{\mathbb{R}^d}g(x,h+s)g(x, s)ds &= \frac{1}{(2\pi)^{d/2}}\int_{\mathbb{R}^d}e^{ih^{\top}u}|\widehat g(x,u)|^2du,
\end{align*}
since
\begin{align*}
\frac{1}{(2\pi)^{d/2}}\int_{\mathbb{R}^d}e^{ih^{\top}u}|\widehat g(x,u)|^2du
&=\frac{1}{(2\pi)^{d/2}}\int_{\mathbb{R}^d}e^{ih^{\top}u}\widehat g(x,u) \overline{\widehat g(x,u)}du\\
&=\frac{1}{(2\pi)^{d/2}}\int_{\mathbb{R}^d}e^{ih^{\top}u}\widehat g(x,u) \widehat g(x,-u)du
\\
&=  \int_{\mathbb{R}^d}g(x,u) g(x,u-h)du =\int_{\mathbb{R}^d}g(x,h+s)g(x, s)ds,
\end{align*}
see e.g.~\citet[Section 23.3.5]{GasquetWitomski1999} for properties of convolutions of $L^2$-Fourier transforms.
Now we define $\gamma(u):=\int_{\mathcal{X}}|\widehat g(x,u)|^2p(dx)$, then 
\begin{align*}
R_X(h) &=   \int_{\mathcal{X}}\left(\frac{1}{(2\pi)^{d/2}}\int_{\mathbb{R}^d}e^{ih^{\top}u}|\widehat g(x,u)|^2du\right) p(dx)\\
&=\frac{1}{(2\pi)^{d/2}}
\int_{\mathbb{R}^d}e^{ih^{\top}u} \left(\int_{\mathcal{X}}|\widehat g(x,u)|^2  p(dx)\right) du
=\frac{1}{(2\pi)^{d/2}}
\int_{\mathbb{R}^d}e^{ih^{\top}u}\gamma(u) du.
\end{align*}
Note that $\int_{\mathbb{R}^d}\gamma(u)du=\int_{\mathbb{R}^d\times \mathcal{X}}|\widehat g(x,u)|^2p(dx)du <\infty$,
hence
 $\gamma \in L^1(\leb)$ is a non-negative function  and  $\gamma^{1/2} \in L^2(\leb)$.
Hence, from the representation above 
we see that $u\mapsto\frac{1}{(2\pi)^{d/2}}\gamma(u)$ is proportional to the  
corresponding spectral density of $X$.
\end{proof}

Next we present a general method for constructing a weight function which can reproduce a given covariance function. 
E.g.~suppose we are given a covariance function $R(h)\in L^1(\leb)$ and we would like to find a function $f\in L^2(\leb)$ such that 
$R(h) = \int_{\mathbb{R}^d}f(h+s)f(s) ds$. 
\begin{proposition}\label{Rhotof}
Suppose $R(h)\in L^1(\leb)$ is a covariance function with spectral density (up to a factor) given by $u\mapsto \gamma(u)$.  
\begin{enumerate}
\item  Suppose that  $\gamma^{1/2}_e= \sqrt{\gamma}$ is the \emph{even root} of $\gamma$. Let $f=f_e$ denote the corresponding $L^2$-Fourier transform of $\gamma_e^{1/2}$. 
 Then  $|\widehat f_e(u)|^2= \gamma(u)$.
\item Suppose that  $\gamma^{1/2}_o=-\sqrt{\gamma}$ is the \emph{odd root} of $\gamma$. Let $f=f_o$ denote the corresponding $L^2$-Fourier transform of $\gamma_o^{1/2}$.  Then 
 $|\widehat f_o(u)|^2= \gamma(u)$. 
\end{enumerate}
In both cases, we have for all $h$ that   
\begin{align*}
R(h) =\int_{\mathbb{R}^d}f(h+s)f(s) ds.
\end{align*} 

\end{proposition}
\begin{proof}
According to Bochner's theorem, the covariance function can be represented as 
$R(h) = \frac{1}{(2\pi)^{d/2}}
\int_{\mathbb{R}^d}e^{ih^{\top}u} \gamma(u) du$,
where $\gamma$ is an even function which is proportional to the corresponding spectral density. Note that we know that $\gamma \in L^1(\leb)$, which does not generally imply that its square root is integrable, but we know  that at least $\gamma^{1/2}\in L^2(\leb)$.
 Then the $L^2$-Fourier transform of $\gamma^{1/2}$ exists and is in the following denoted by $f$, i.e.~$f = \widehat {\gamma^{1/2}}$. 
 It is a well-known result, see e.g.~\citet[Proposition 22.2.1]{GasquetWitomski1999} that $\widehat f(u) = \gamma^{1/2}(-u)$ (a.e.) for all $x$. 
 
Suppose now that  $\gamma^{1/2}_e= \sqrt{\gamma}$. Then $\gamma^{1/2}_e(-u)= \sqrt{\gamma(-u)}= \sqrt{\gamma(u)}=\gamma^{1/2}_e(u)$ is an \emph{even root} of $\gamma$, then $f=f_e$ is even, too.  Then $\widehat f_e(u)=\gamma^{1/2}(-u)= \gamma^{1/2}(u)$ and $|\widehat f_e(u)|^2= \gamma(u)$.
Similarly, when  $\gamma^{1/2}_o=-\sqrt{\gamma}$, then $\gamma^{1/2}_o(-u)= -\sqrt{\gamma(-u)}= -\sqrt{\gamma(u)}=-\gamma^{1/2}_o(u)$ is an \emph{odd root} of $\gamma$, and $f=f_o$ is odd, too.  Then $\widehat f_o(u)=-\gamma^{1/2}(-u)= -\gamma^{1/2}(u)=\gamma_o^{1/2}(u)$ and $|\widehat f_o(u)|^2= \gamma(u)$. In both cases, we have for all $h$ that   \begin{align*}
R(h) =\frac{1}{(2\pi)^{d/2}}
\int_{\mathbb{R}^d}e^{ih^{\top}u} |\widehat f(u)|^2  du =\int_{\mathbb{R}^d}f(h+s)f(s) ds.
\end{align*} 
\end{proof}
Under stronger $L^1(\leb)$-integrability conditions on the weight function, a related result can be found in \cite{EmilThesis}.

Also, we can deduce the following result on the relation between a volatility modulated mixed moving average and a volatility modulated moving average process.
\begin{corollary}\label{corollary1}
Define the volatility modulated moving average field 
 \begin{align*}
Z=(Z_t)_{t\in \mathbb{R}^d}=\left(\int_{\mathbb{R}^d}f(t-s)\sigma(s)W(ds) \right)_{t\in \mathbb{R}^d},
\end{align*}
where the weight function $f$ is proportional to the $L^2$-Fourier transform of $\gamma^{1/2}$, where  
$\gamma(u):=\int_{\mathcal{X}}|\widehat g(x,u)|^2p(dx)$ and, for $u \in \mathbb{R}$,  $\widehat g(x, u)$ denotes the $L^2$-Fourier-transform of $g(x, \cdot)$ and $W$  denotes a Brownian motion.

Then the VMMMA field $(X_t)_{t\in\mathbb{R}^d}$ and the VMMA field  $(Z_t)_{t\in\mathbb{R}^d}$ have the identical correlation structure. Moreover, in the absence of stochastic volatility, the VMMMA and the VMMA field have identical finite dimensional distributions.
\end{corollary}
Note that the latter result has already been mentioned in  \citet[p.548--549]{Surgailisetal1993} and the former result is a direct consequence of our above derivations.

In this section, we have shown how one can construct a suitable  weight function for a VM(M)MA process which can reproduce a given correlation structure. However, it is important to note that the  
covariance function does not specify the weight function uniquely. Also, the covariance function  does not give any indication about whether  the underlying random field is a mixed moving average or just a moving average field.

\subsection{Starting from a stochastic partial differential equation}
An interesting alternative to modelling via the second-order structure is to find a suitable (stochastic) differential equation which can describe the empirical object under investigation. 

To simplify the exposition, let us in the following focus exclusively on VMMA fields (rather than on VMMMA fields). 
\cite{BNBV2011} have recently discussed the relation between certain types of VMMA fields and solutions to certain types of stochastic partial differential equations. One of their key results -- subject to appropriate regularity conditions --    was that  if 
the weight function $g$ is chosen to be a Green's function of a certain type of an  SPDE, then the resulting MA can be regarded as a mild solution to the SPDE. This suggests that one can generate a wide class of VMMA by choosing various Green's functions as a weight function. We demonstrate this approach in the following when $d=2$, i.e.~when we consider random fields on the plane.
\subsubsection{Examples: Modelling the weight  and  correlation functions in the plane}\label{ginplane}
Let us study some examples of possible weight and correlation functions, where we focus on the case when $d=2$, i.e.~we consider random fields in the plane. 

While mixed-moving average fields aim to model a tempo-spatial objective directly, the choice of the kernel function can sometimes be motivated from certain types of stochastic (partial) differential equations. 
In this context, we revisit the work by \cite{Whittle1954} and  by \cite{Heine1955} in particular.

In the planar case, there are three types of second-order stochastic partial differential equations (SPDEs) which are relevant for model building, depending whether both axes have space-like or time-like features or whether one axes has time-like and one space-like behaviour. In the context of SPDEs the so-called Green's function plays a key role. As discussed in 
\cite{BNBV2011}, under appropriate regularity conditions it is possible to link ambit fields, which are special cases of moving average fields, where the weight function is given by a Green's function, to mild solutions of SPDEs.  Following these findings we might want to choose the weight function in the moving average field as 
\begin{align*}
f(z_1, z_2) = G(z_1, z_2),
\end{align*}
where $G(z_1, z_2)$ is a Green's function.
Also, let us fix the notation we use in the context of SPDEs. We consider SPDEs of the type
\begin{align*}
L\left(\frac{\partial}{\partial z_1}, \frac{\partial}{\partial z_2}\right) Z(z_1, z_2) = \epsilon(z_1, z_2),
\end{align*}
where $Z$ denotes a two-parameter stochastic process and $\epsilon$ denotes a stochastic noise term; both $Z$ and $\epsilon$ are assumed to have zero mean. 
Under suitable regularity conditions, one can then express the solution formally in terms of a Green's function, i.e.
\begin{align*}
Z(t_1, t_2) = \int_{\R}\int_{\R} G(t_1-s_1, t_2-s_2) \epsilon(s_1, s_2) ds_1 ds_2, 
\end{align*}
where the Green's function satisfies the following equation
\begin{align*}
L\left(\frac{\partial}{\partial z_1}, \frac{\partial}{\partial z_2}\right)G(z_1, z_2) = \delta(z_1) \delta(z_2),
\end{align*}
where $\delta$ denotes the Dirac delta function.

We start off with an example which is motivated by having one time-like and one space-like axis. 
\begin{example}[Parabolic case] Consider the parabolic  SPDE of the type 
\begin{align*}
L\left(\frac{\partial}{\partial z_1}, \frac{\partial}{\partial z_2}\right)=\left(\frac{\partial }{\partial z_1} + \alpha \right)^2-\gamma^2 \left(\frac{\partial }{\partial z_2} + \beta \right) \quad \text{for } 0\leq \alpha^2 < \beta \gamma^2,
\end{align*}
 which corresponds to a 
  space-like $z_1$-axis and a time-like $z_2$-axis.
According to \citet[formula (5.3)]{Heine1955} the corresponding Green's function is given by
\begin{align*}
G(z_1,z_2) = \frac{-1}{2\gamma \sqrt{\pi z_2}}\exp\left(-\alpha z_1 -\beta z_2 - \frac{z_1^2 \gamma^2}{4z_2} \right)U(z_2),
\end{align*}
where
$U(z_2) = \mathbb{I}_{\{z_2\geq 0 \}}$.
The correlation function associated with such a Green's function is  according to \citet[formula (5.10)]{Heine1955} 
given by \begin{align*}
\rho(z_1, z_2) &= \rho(-z_1, -z_2), \text{where } z_2 > 0\\
&= \frac{e^{-2AB}}{\sqrt{\pi}}\int_{-\infty}^{A-B}e^{-t^2}dt +
\frac{e^{2AB}}{\sqrt{\pi}}\int^{\infty}_{A+B}e^{-t^2}dt, 
\end{align*}
where $A= \frac{\gamma}{2\sqrt{z_2}}(z_1 + \frac{2\alpha z_2}{\gamma^2})$ and $B= \sqrt{z_2 (\beta - \frac{\alpha^2 z_2}{\gamma^2})}$.
\end{example}

The above examples suggests that one could  for instance study a volatility modulated MA of the form
\begin{multline*}
X(t_1, t_2)=\int_{\times \R}\int_{-\infty}^{t_2} \frac{-1}{2\gamma \sqrt{\pi (t_2-s_2)}}\exp\left(-\alpha (t_1-s_1) -\beta (t_2-s_2) - \frac{(t_1-s_1)^2 \gamma^2}{4(t_2-s_2)} \right)\\
\cdot\sigma(s_1, s_2) W(ds_1, ds_2),
\end{multline*}
where we use exactly the same notation as in the Example above. 
Here our weight function is parameterised using three parameters $\alpha, \beta, \gamma$. Now we could generalise our model by extending the MA field to a MMA field by randomising one (or more) of the parameters. For instance, if we randomise the parameter $\gamma$ we change the behaviour of the process in the time dimension.
It should be noted that while \cite{BNBV2011} discussed which types of ambit fields can be considered as mild solutions to SPDEs, we only use the SPDE as a starting point for generating interesting weight functions. In a next step, we might want to generalise these functions, by randomising some of the corresponding parameters.  

Next, let us study  an example with two space-like axes. 
\begin{example}[Elliptic case]
Consider an elliptic SPDE given by 
\begin{align*}
L\left(\frac{\partial}{\partial z_1}, \frac{\partial}{\partial z_2}\right)=\left(\frac{\partial}{\partial z_1} - \alpha \right)^2 + \frac{\partial^2}{\partial z_2^2} - \gamma^2, \quad \text{ for }0\leq \alpha < \gamma,
\end{align*}
which corresponds to two space-like axes.
According to \citet[formulae (6.5) and (6.6)]{Heine1955}
the corresponding Green's function is given by 
\begin{align*}
G(z_1,z_2)= \frac{e^{\alpha z_1}}{2\pi}K_0(\gamma r),
\end{align*}
where $r = \sqrt{z_1^2+z_2^2}$ and $K$ is the modified Bessel function of the second kind.
In the case $\alpha = 0$, we have 
\begin{align*}
\rho(z_1,z_2) = \gamma r K_1(\gamma r).
\end{align*}
More generally, 
\begin{align*}
\rho(z_1,z_2) = \frac{\sqrt{\gamma^2-\alpha^2}}{\sin^{-1}(\alpha/\gamma)}\int_{z_1}^{\infty} \sinh(\alpha \tau)K_0(\gamma\sqrt{\tau^2+z_2^2}) d\tau.
\end{align*}
\end{example}

Finally we focus on an example with two time-like axes.
\begin{example}[Hyperbolic case]
Consider a hyperbolic SPDE
\begin{align*}
L\left(\frac{\partial}{\partial z_1}, \frac{\partial}{\partial z_2}\right)=\left( \frac{\partial }{\partial z_1 } +\alpha\right) \left( \frac{\partial }{\partial z_2 } +\beta\right) + \gamma^2, \quad \text{ for } \alpha > 0, \beta > 0, \gamma^2\geq 0,
\end{align*}
which corresponds to two time-like axes with $z_1, z_2\geq 0$. Due to \cite[formulae (7.2) and (7.7)]{Heine1955}, the corresponding Green's function is given by 
\begin{align*}
G(z_1,z_2) = e^{-\alpha z_1-\beta z_2}J_0(2\gamma\sqrt{z_1z_2})U(z_1)U(z_2),
\end{align*}
where $J_0$ is the zero order regular Bessel function of the first kind,
which implies the following correlation function in the first quadrant:
\begin{align}\label{rhoZ}
\rho(Z_1,Z_2) = \frac{2\alpha \beta}{\sqrt{\alpha^2+\beta^2}}\sqrt{1+\frac{\gamma^2}{\alpha \beta}} \int_{P(Z_1,Z_2)}^{\infty}e^{-\alpha z_1-\beta z_2}J_0(2\gamma\sqrt{z_1z_2})d\zeta,
\end{align}
where the precise definition of the bound $P(Z_1,Z_2)$ is given in \cite{Heine1955}. Note that the integration in \eqref{rhoZ} is done in the $\zeta$-direction, which has  an angle of  $\theta=\tan^{-1}(\alpha/\beta)$ with the $z_1$-axis.
In the special case when  $\gamma = 0$, we have
\begin{align*}
G(z_1,z_2)&= e^{-\alpha z_1 - \beta z_2}U(z_1)U(z_2),\\
\rho(z_1,z_2) &= \exp(-\alpha |z_1|- \beta |z_2|).
\end{align*}
\end{example}

Overall, we conclude this section by noting that an SPDE-based approach where one uses certain types of Green's functions can be  useful in order to find relevant candidates  of weight functions for VMMA fields. In order to extend this framework and to construct suitable weight functions for the mixed, i.e.~the VMMMA, case, 
one could then consider randomising some or all of the parameters appearing in the weight functions of an VMMA field through a \Levy basis. This approach will give rise to a broad class of weight functions which can be used in the context of the more general VMMMA fields. In particular, we are not restricted to studying classes of random fields which appear as (mild) solutions to particular SPDEs, but can go beyond such classes if this becomes necessary in a particular application.

\subsection{Remarks}
Specifying the weight function in the (M)MA representation is clearly not enough for characterising the entire  VM(M)MA model. In fact, the weight function only determines the  second order properties of the model. As soon as stochastic volatility is present, one needs to specify a suitable weight function in the volatility process as well. In order to identify the functional form of the weight function in the latent stochastic volatility field, one needs to consider higher moments as well as we have seen in \eqref{CovX2}.
Alternatively, one might want to find a good proxy, see e.g.~\cite{Reveillac2009} and \cite{Pakkanen2013} for some work along those lines,  to estimate stochastic volatility and  to infer the corresponding properties of the volatility field from such a proxy. With a good volatility estimator at hand, one could essentially repeat the procedure above, but now applied to the estimated volatility field, and  specify the weight function in the volatility process. Constructing such a volatility proxy in a general tempo-spatial setting  will be an interesting area for future research.

\section{Multi-self-similar random fields with stochastic volatility}\label{SectMSS}
Let us now turn to constructing (multi-) self-similar random fields which exhibit stochastic volatility. 
Self-similarity is an important concept in probability: If a stochastic process is self-similar,  it is similar to a part of itself, meaning that it is invariant under scaling in time and space. \cite{Lamperti1962} showed that self-similarity is connected to limit theorems and hence this concept has attracted a lot of interest in  stochastic modelling.  For instance, \cite{Taqqu1986} and \cite{STaqqu1994} mention applications in communications, economics, geophysics, hydrology and turbulence
and 
\cite{GPT2007} give a recent account on the importance of self-similarity in applications in climatological and environmental sciences. In the following, we will use a generalised Lamperti transform to construct (multi-) self-similar random fields  which account for stochastic volatility. Moreover, we discuss the important subclass of  multi-self-similar random fields  with second-order stationary increments.

\subsection{Definitions and generalised Lamperti transform}
Self-similarity for random fields is typically defined as follows, cf.~\cite[p.~392]{STaqqu1994}.
\begin{definition}
A random field $(X(t))_{t\in \mathbb{R}^d}$ is \emph{self-similar with index} $H> 0$ if $(X(at))_{t\in \mathbb{R}^d}$ and $(a^H X(t))_{t\in \mathbb{R}^d}$ have the same finite dimensional distributions for all $a > 0$.
\end{definition}
The above definition is a direct generalisation from the one-parameter case (when $d=1$). However,  in the context of random fields a more refined definition of self-similarity is needed which allows for component-wise self-similarity and hence \cite{GPT2007} introduced the concept of multi-self-similarity. More precisely, we have the following definition, cf.~\citet[p.~401]{GPT2007}. Let $\mathbb{R}^d_+$ denote the $d$-fold Cartesian product $\mathbb{R}_+\times \cdots \times \mathbb{R}_+$.
\begin{definition}
A random field $(X(t))_{t\in \mathbb{R}^d}$ is \emph{multi-self-similar with index} $H=(H_1,\dots, H_d)^{\top}\in \mathbb{R}^d_+$ ($H$-mss) if $(X(a_1t_1,\dots, a_d t_d))_{t\in \mathbb{R}^d}$ and $(a_1^{H_1}\cdots a_d^{H_d} X(t))_{t\in \mathbb{R}^d}$ have the same finite dimensional distributions for all $a_1, \dots, a_d > 0$.
\end{definition}
Note that in the case when $a_1=\cdots=a_d = a > 0$ and $H_1+\cdots +H_d = H > 0$, multi-self-similarity reduces to the classical self-similarity. Also, \cite{GPT2007} point out that the definition of multi-self-similarity depends on the coordinate system used to parametrise the random field.

It is well-known that in the one-parameter case, the Lamperti transform, cf.~\cite{Lamperti1962}, can be used to construct a self-similar process from a strictly stationary process. In the following, we will use the generalised Lamperti transform as derived in \citet[p.~401]{GPT2007}.
Let us define $Y=(Y(t))_{t\in \mathbb{R}^d_+}$ with
\begin{align}\label{ssY}
Y(t) = \left(\prod_{j=1}^dt_j^{H_j}\right) X(\log(t)), \quad \text{ with } \log(t) = (\log(t_1), \dots, \log(t_n))^{\top} \text{ for } t\in \mathbb{R}^d_+.
\end{align}
Then $Y$ is $H$-mss. Conversely, 
for an $H$-mss process $Y$, $X=(X(t)_{t\in \mathbb{R}^d}$ given by
\begin{align*}
X(t) = e^{-t^{\top}H}Y(e^t), \quad \text{ with }e^t= (e^{t_1}, \dots, e^{t_d})^{\top}  \text{ for } t\in \mathbb{R}^d
\end{align*}  
is (strictly) stationary, cf.~\citet[Proposition 2.1.1]{GPT2007}.
We now obtain the following  result. 
\begin{proposition}\label{SSProp}
Let $X=(X(t))_{t\in \mathbb{R}^d}$ as defined  \eqref{VMMMA}, where the volatility field is assumed to satisfy Condition \ref{V2}, and let  $Y=(Y(t))_{t \in \mathbb{R}^d_+}$ be as defined in \eqref{ssY}. Then
we have the following results.
\begin{enumerate}
\item $Y$ is 
an $H$-mss random field whose distribution is characterised by
\begin{align}\label{Hmsschar}
\mathbb{E}(i \theta Y(t)) = \exp\left(\Lambda_V\left(\frac{\theta^2}{2}\prod_{j=1}^dt_j^{2 H_j} \right) \right).
\end{align}
\item  Under Condition \ref{I}, $Y$ has  finite second moment and we have  for $t, t^* \in \mathbb{R}^d_+$:
\begin{align*}
\Cov(Y(t), Y(t^*)) = \exp\left(2H^{\top}\left(\frac{\log(t)+\log(t^*)}{2}\right) \right) R_X(\log(t)-\log(t^*)).
\end{align*}
\end{enumerate}
\end{proposition}

\begin{proof}
 $Y$ is an $H$-mss random field according to  \citet[Proposition 2.1.1]{GPT2007}. The results for the characteristic function and the covariance function follows by direct calculation using  the stationarity of $X$.
\end{proof}

From the structure of the covariance function and recalling that $X$ is strictly stationary, we immediately get that the process $Y$
 belongs to the class of so-called \emph{locally stationary reducible} random fields, cf.~\citet[p.~403]{GPT2007} and also \cite{GentonPerrin2004}.

Similar to the findings in \cite{BNPerezAbreu1999} in the case when $d=1$ we deduce from \eqref{Hmsschar} the following result for general $d \in \N$.
\begin{proposition}
Under the conditions of Proposition \ref{SSProp}, for every $t\in \mathbb{R}^d_+$ the law of $Y(t)$ is of type G.
\end{proposition}
\begin{proof}
Let $t\in \mathbb{R}^d_+$. 
From \eqref{Hmsschar} we have
$\mathbb{E}(i \theta Y(t)) = \exp(\Lambda_V(\frac{\theta^2}{2}\prod_{j=1}^dt_j^{2 H_j} ) ) =  \exp(- \Phi(\frac{\theta^2}{2}))$,
where
$\Phi(\zeta) :=-\Lambda_V(\zeta \prod_{j=1}^dt_j^{2 H_j})$.
Since $\prod_{j=1}^dt_j^{2 H_j}\geq 0$, one easily finds that $\Phi(0)=0$ and $\Phi'$ is  completely monotone  on $(0,\infty)$. This implies that the law of $Y(t)$ is of type G, cf.~\citet[Proposition 3]{Rosinski1991}.
\end{proof} 

Summing up we have found  a method for  constructing multi-self-similar random fields which have type G distribution and allow for stochastic volatility.

\subsection{Self-similar processes with translation invariant increments}
In applications  one is sometimes  interested in (multi)-self-similar processes with stationary increments. In the following we will show how such processes can be constructed. 

Note that we use the term \emph{stationary} increments in the context of random fields interchangeably with saying  that a  random field is invariant under translation.
\begin{definition}
The random field $(Y(t))_{t \in \mathbb{R}_+^d}$ has \emph{second-order translation-invariant increments}  if for all $t, h \in \mathbb{R}_+^d$, we have that  $\mathbb{E}(Y(t+h)-Y(t))^2=\mathbb{E}(Y(h))^2$.
\end{definition}
 The following results are related to the findings in  \citet[Section 4.1]{BNPerezAbreu1999}. However, the difference in our set-up is that we work with volatility modulated Gaussian fields rather than with (one-parameter) L\'{e}vy processes as the driving process of the MMA fields, respectively.
\begin{proposition}\label{CorTransInvIncr}
Let $Y=(Y(t))_{t\in \mathbb{R}^d}$ be as defined in \eqref{ssY} and suppose that $Y$ has second-order translation-invariant increments. Then 
\begin{align*}
\Cov(Y(t),Y(s))&= \frac{1}{2} \left[\prod_{j=1}^dt_j^{2 H_j}+\prod_{j=1}^ds_j^{2 H_j}-\prod_{j=1}^d(t_j-s_j)^{2 H_j} \right] \Var(X(0)), \quad s, t \in \mathbb{R}_+^d.
\end{align*}
Moreover, the corresponding  correlation function of $X$ is  given by
\begin{align}\label{rhoX}
\rho_X(h)
&= \cosh(h^{\top}H)-2^{2\sum_{j=1}^dH_j-1}\prod_{k=1}^d\sinh^{2H_k}\left(\frac{1}{2}h_k \right), \quad h \in \mathbb{R}^d.
\end{align}
\end{proposition}

\begin{proof}
 For all $t \in \mathbb{R}^d_+$ we have 
$\mathbb{E}(Y(t))= 0$ and  $\Var(Y(t))= \mathbb{E}(Y^2(t))= \prod_{j=1}^dt_j^{2 H_j} \Var(X(0))$.
Under the assumption that  $Y$ has second-order translation invariant increments, we get for $s, t \in \mathbb{R}^d_+$ that 
$\mathbb{E}(Y(t)-Y(s))^2 =\mathbb{E}(Y^2(t-s))= \prod_{j=1}^d(t_j-s_j)^{2 H_j} \Var(X(0))$. Moreover
\begin{align*}
\Cov(Y(t),Y(s))&= \mathbb{E}(Y(t)Y(s))=\frac{1}{2}\left[\mathbb{E}(Y^2(t))+\mathbb{E}(Y^2(s))-\mathbb{E}(Y(t)-Y(s))^2 \right]\\
&= \frac{1}{2} \left[\prod_{j=1}^dt_j^{2 H_j}+\prod_{j=1}^ds_j^{2 H_j}-\prod_{j=1}^d(t_j-s_j)^{2 H_j} \right] \Var(X(0)).
\end{align*}
For the corresponding process $X$, we get using the inverse Lamperti transform for $h,0 \in \mathbb{R}^d$:
\begin{align*}
\Cov(X(h),X(0))&=\mathbb{E}(X(h)X(0)) = e^{-h^{\top}H} e^{-0^{\top}H} \Cov(Y(e^h), Y(e^0))\\
&= e^{-h^{\top}H}   \frac{1}{2} \left[\prod_{j=1}^de^{2 h_j  H_j}+1-\prod_{j=1}^d(e^{h_j}-1)^{2 H_j} \right] \Var(X(0)).
\end{align*}
Hence
\begin{align*}
\Cor(X(h),X(0))&= \frac{1}{2}\left[ e^{h^{\top}H}+e^{-h^{\top}H} -\prod_{j=1}^de^{-h_jH_j}\left(e^{h_j}-1\right)^{2 H_j}\right]
\\
&= \frac{1}{2}\left[ e^{h^{\top}H}+e^{-h^{\top}H} -\prod_{j=1}^d\left(e^{h_j/2}-e^{-h_j/2}\right)^{2 H_j}\right]
\\
&= \frac{1}{2}\left[ e^{h^{\top}H}+e^{-h^{\top}H} -\prod_{j=1}^d\left(e^{h_j/2}-e^{-h_j/2}\right)^{2 H_j}\right]\\
&=\left[\cosh(h^{\top}H)-2^{2\sum_{j=1}^dH_j-1}\prod_{k=1}^d\sinh^{2H_k}\left(\frac{h_k}{2} \right)\right].
\end{align*}
\end{proof}

\subsubsection{Spectral density and weight function in case $d=1$}
In the following, we concentrate on the case when $d=1$. In this case, \emph{translation-invariant} and \emph{stationary} increments are equivalent expressions.  Then the results from Proposition \ref{CorTransInvIncr} simplify as follows.

\begin{corollary}
Let $Y=(Y(t))_{t\in \mathbb{R}}$ be as defined in \eqref{ssY} (for $d=1$), i.e.~$Y(t)=t^HX(\log(t))$ for $H, t > 0$ and suppose that $Y$ has second-order stationary increments. Then 
\begin{align*}
\Cov(Y(t),Y(s))&= \frac{1}{2} \left[t^{2 H}+s^{2 H}-(t-s)^{2 H} \right] \Var(X(0)), \quad s, t \in \mathbb{R}_+^d.
\end{align*}
Moreover, the correlation function of $X$ is given by
\begin{align}\label{rhoXd1}
\rho_X(h)
&= \cosh(h H)-2^{2 H-1}\sinh^{2H}\left(\frac{1}{2}h \right), \quad h \in \mathbb{R}.
\end{align}
\end{corollary}
Based on the result for the correlation function of $X$ given in \eqref{rhoXd1}, we can derive an explicit formula for the spectral density associated with the correlation function. More precisely, we have the following result.

\begin{proposition}\label{PropSpectralDensity}
Suppose that $d=1$ and that $0<H<1$.
\begin{enumerate} 
\item The spectral density associated with the correlation function \eqref{rhoX}  exists and 
is for $w\in \mathbb{R}$ given by
\begin{align}\label{spectraldensity}
\gamma(w) = \frac{1}{2\pi}\sum_{k=0}^{\infty}{ 2H\choose k}(-1)^{k-1} \frac{k-H}{(k-H)^2+w^2}
\end{align}
In the special case when  $H=1/2$, we get $\gamma(w)=2\pi^{-1}(1+(2w)^2)^{-1}$.
\item Consider the moving average representation derived in  Corollary \ref{corollary1}. The corresponding weight function $f$ in that representation is proportional to the $L^2$-Fourier transform of $\sqrt{\gamma}$.
 \end{enumerate}
\end{proposition}
\begin{proof}
1.~The first part of the proposition has been shown in \citet[Theorem 4]{BNPerezAbreu1999} in the context of L\'{e}vy processes. We give a short proof in the following, where we see that the key arguments do not change in our different modelling set-up. 
 In order to guarantee the existence of a spectral density, we need that $\int_{\mathbb{R}^d} |\rho_X(h)|dh < \infty$, cf.~\cite{Bochner1955}, which holds as soon as $0<H<1$. 
Recall that the spectral density function satisfies
$\rho_X(h) = \int_{\mathbb{R}}e^{ih w}f_X(w) dw$.
Hence, using Fourier inversion, we have
 $f_X(w) = (2\pi)^{-1}\int_{\mathbb{R}}e^{-ihw}\rho_X(h)dh$.
 
 First, we rewrite the correlation function and then we use the Fourier inversion for each summand. In particular, we have
 \begin{align*}
 \rho(h) &= \frac{1}{2}\left[e^{h H} + e^{-h H} -\left(e^{h/2}-e^{-h/2}\right)^{2 H}\right]
 &= \frac{1}{2}\left[e^{-h H} +\sum_{k=1}^{\infty}{ 2H\choose k}(-1)^{k-1} e^{-h(k-H)}\right].
 \end{align*}
 Since the  correlation function is symmetric, we may write $\rho_{X}(h)= \rho_X(|h|)$. 
Hence $\rho(h) = \rho_1(h)+\rho_2(h)$, where for $h\in \mathbb{R}$
\begin{align*}
\rho_1(h):= \frac{1}{2}e^{-|h| H},
&&
\rho_2(h):=\sum_{k=1}^{\infty}{ 2H\choose k}(-1)^{k-1} e^{-|h|(k-H)}.
\end{align*}
Corresponding to the two parts of the correlation function we split the spectral density in two parts and have $f_X(w) = f_1(w)+f_2(w)$ for $w \in \mathbb{R}$ , where
\begin{align*}
f_1(w)&=(2\pi)^{-1}\int_{\mathbb{R}}e^{-ihw}\frac{1}{2}e^{-|h|H}dh
=  \frac{1}{2} (2\pi)^{-1}\int_{\mathbb{R}}e^{-ihw}e^{-|h|H}dh 
=\frac{1}{2\pi} \frac{H}{H^2+w^2},
\end{align*}
since the  Fourier inversion of the characteristic function of the Cauchy density leads to the Cauchy density. The second part uses the same arguments and combining the two terms  gives the final result.\\
 2.~The proof of the second part of the proposition follows along the lines of 
 the proof of Proposition \ref{Rhotof} where we worked with Fourier transforms and inversions in $L^2$. Altogether, we get that the weight function $f$ in the VMMA representation is proportional to the $L^2$-Fourier transform of the square root of the spectral density, i.e.~$\sqrt{\gamma}$.
 \end{proof}

We conclude this Section with a characterisation result for our stochastic volatility modulated processes. 
 
\begin{corollary}
Suppose $X$ has spectral density given by \eqref{spectraldensity}. Then its correlation function  is given by 
$\rho_X(h)= \cosh(hH)-2^{2H-1}\sinh^{2H}(h/2)$,  for $h \in \mathbb{R}$.
 Moreover,  $Y=(Y(t))_{t>0}$ defined by
$Y(t)=t^HX(\log(t))$ is $H$-self-similar and  has second-order stationary increments.
\end{corollary}

Interestingly, the above result  is exactly the same as obtained in  
 \citet[Theorem 4]{BNPerezAbreu1999} in the context of \levy-driven stationary and self-similar processes.

\section{Conclusion}\label{SectCon}
This paper has focused on mixed moving average fields which allow for stochastic volatility modulation.  
We have studied the probabilistic properties of such processes in detail and have in particular focused on volatility modulated mixed moving average  fields whose marginal distribution is of type G. Such processes are relevant in a wide range of applications and constitute an important extension of  mixed moving average fields which are constructed based on type G \Levy bases and do not account for stochastic volatility. 

Moreover, we have introduced two methods which can be used for finding suitable weight functions in the moving average representation. One is based on the idea of starting from a suitable integrable covariance function and modelling the weight function as the $L^2$-Fourier transform of a root of the corresponding spectral density. The other approach uses Green's functions from SPDEs as input.
 
Another  contribution of this paper is that it provides a tractable method for constructing multi-self-similar random fields which allow for stochastic volatility and have type G distribution.
 
Finally, in the one-parameter case when $d=1$, we have constructed a class of stochastic processes which has the aforementioned properties of self-similarity, stochastic volatility, a type G distribution and, at the same time, has stationary increments - a property often required in empirical work.

\section*{Acknowledgement} 
A.~E.~D.~Veraart acknowledges financial support by 
a Marie Curie
FP7 Career Integration Grant within the 7th European Union Framework Programme.

\bibliographystyle{agsm}
\bibliography{AlmutBib}

\begin{appendix}

\end{appendix}
\end{document}